\documentclass{amsart}

\input xy
\xyoption{all}

\newbox\noforkbox \newdimen\forklinewidth
\setbox0\hbox{$\textstyle\smile$}
\setbox1\hbox to \wd0{\hfil\vrule width \forklinewidth depth-2pt
 height 10pt \hfil}
\setbox\noforkbox\hbox{\lower 2pt\box1\lower 2pt\box0\relax}
\def\unionstick{\mathop{\copy\noforkbox}\limits}
\def\nonfork_#1{\unionstick_{\textstyle #1}}

\setbox0\hbox{$\textstyle\smile$}
\setbox1\hbox to \wd0{\hfil{\sl /\/}\hfil}
\setbox2\hbox to \wd0{\hfil\vrule height 10pt depth -2pt width
              \forklinewidth\hfil}
\newbox\doesforkbox
\setbox\doesforkbox\hbox{\lower 2pt\box1 \lower 2pt\box2\lower2pt\box0\relax}
\def\nunionstick{\mathop{\copy\doesforkbox}\limits}

\def\fork_#1{\nunionstick_{\textstyle #1}}

\newcommand{\dnf}{\unionstick}
\newcommand{\nf}{\unionstick}
\newcommand{\dnfb}[4]{#2 \overset{#4}{\underset{#1}{\overline{\nf}}} #3}

\newbox\noforkboxs \newdimen\forklinewidths
\setbox0\hbox{$\textstyle\smile$}
\setbox1\hbox to \wd0{\hfil\vrule width \forklinewidths depth-2pt
 height 10pt \hfil}
\setbox\noforkboxs\hbox{\lower 2pt\box1\lower 2pt\box0\relax}
\def\unionsticks{\mathop{\copy\noforkboxs^{\mkern-8mu\mathrm{s}}}\limits}
\def\nonforks_#1{\unionsticks_{\textstyle #1}}

\setbox0\hbox{$\textstyle\smile$}
\setbox1\hbox to \wd0{\hfil{\sl /\/}\hfil}
\setbox2\hbox to \wd0{\hfil\vrule height 10pt depth -2pt width
              \forklinewidth\hfil}
\newbox\doesforkboxs
\setbox\doesforkboxs\hbox{\lower 2pt\box1 \lower 2pt\box2\lower2pt\box0\relax}
\def\nunionsticks{\mathop{\copy\doesforkbox}\limits}

\def\forks_#1{\nunionsticks_{\textstyle #1}}

\newcommand{\dnfs}{\unionsticks}
\newcommand{\nfs}{\unionsticks}
\newcommand{\dnfbs}[4]{#2 \overset{#4}{\underset{#1}{\overline{\nfs}}} #3}
\usepackage{amsmath}
\usepackage{amssymb}

\usepackage{amsthm}  

\newtheorem{same}{This should never appear}[section]
\newtheorem{defin}[same]{Definition}

\newtheorem{remark}[same]{Remark}
\newtheorem{theorem}[same]{Theorem}
\newtheorem{example}[same]{Example}
\newtheorem{lemma}[same]{Lemma}
\newtheorem{fact}[same]{Fact}
\newtheorem{question}[same]{Question}
\newtheorem{cor}[same]{Corollary}

\newtheorem{nota}[same]{Notation}

\newtheorem*{theorem-2}{Theorem \ref{sta-2}}
\newtheorem*{theorem-3}{Theorem \ref{equiv}}
\newtheorem*{cor-1}{Corollary \ref{abs}}
\newtheorem*{theorem-4}{Theorem \ref{inj-equiv}}

\newtheorem{defin*}{Definition}
\newtheorem*{theorem*}{Theorem}

\makeatletter
\newcommand{\skipitems}[1]{%
  \addtocounter{\@enumctr}{#1}%
}

\newcommand{\rest}{\mathord{\upharpoonright}}
\newcommand{\id}{\textrm{id}}

\newcommand\Act{\operatorname{\bf Act}}
\newcommand{\Kk}{\mathbf{K}}
\newcommand\cx{\mathcal {X}}

\newcommand\Set{\operatorname{\bf Set}}

\newcommand{\LS}{\operatorname{LS}}

\newcommand{\leap}[1]{\le_{#1}}

\newcommand{\lea}{\leap{\Kk}}

\newcommand\Acts{\operatorname{\bf Act}}\newcommand\aact{S\text{-}\Acts}
\newcommand{\gtp}{\mathbf{gtp}}
\newcommand{\gS}{\mathbf{gS}}

\newcommand\cs{\mathcal {C}}

\title{On the abstract elementary class of acts with embeddings}

\author{Marcos Mazari-Armida}

\thanks{The first author's research was partially supported by an NSF grant DMS-2348881, a Simons Foundation grant MPS-TSM-00007597 and an AMS-Simons Travel Grant 2022--2024.}

\address{\newline Marcos Mazari-Armida \newline Department of Mathematics \newline Baylor University \newline Waco, Texas, USA}
\urladdr{https://sites.baylor.edu/marcos\_mazari/}
\email{marcos\_mazari@baylor.edu}

\author[ Ji\v{r}\'{\i} Rosick\'{y}]
{ Ji\v{r}\'{\i} Rosick\'{y}}
\thanks{The second author's research was supported by the Grant Agency of the Czech Republic under the grant 22-02964S} 
\address{
\newline  Ji\v{r}\'{\i} Rosick\'{y}\newline
Department of Mathematics and Statistics,\newline
Masaryk University, Faculty of Sciences,\newline
Kotl\'{a}\v{r}sk\'{a} 2, 611 37 Brno,\newline
Czech Republic}
\email{rosicky@math.muni.cz}

\begin{document}

\begin{abstract} We study the class of acts with embeddings as an abstract elementary class.  We show that the class is always stable and show that superstability in the class is characterized algebraically via weakly noetherian monoids.

The study of these model-theoretic notions and limit models leads us to introduce parametrized weakly noetherian monoids and find a characterization of them via parametrized injective acts. Furthermore, we obtain a characterization of weakly noetherian monoids via absolutely pure acts extending a classical result of ring theory.

The paper is aimed at  algebraists and model theorists so an effort was made to provide the background for both.

\end{abstract}


\maketitle

{\let\thefootnote\relax\footnote{{AMS 2020 Subject Classification:
Primary:  03C60, 20M30. Secondary: 03C45, 03C48, 20M12, 20M50 

}}}

\tableofcontents

\section{Introduction}

Abstract elementary classes (AECs for short) constitute   a model-theoretic framework  introduced by Shelah \cite{sh88} to study  classes of structures which are not first-order axiomatizable. The framework has been significantly developed in the past few decades, see for example \cite{grossberg2002}, \cite{shelahaecbook}, \cite{baldwinbook09}, \cite{bova-tame}. For many years, the focus was solely on developing the abstract theory but recently there has been a push toward finding interactions and applications of the framework to algebra. More precisely, a very fruitful interaction between AECs and module theory has been uncovered, see for example \cite{m3}, \cite{m4}, \cite{mj},  \cite{satr-24}, \cite{mlim}, and \cite{bon-mod}. The objective of this paper is to determine if such interactions can occur outside of module theory. In particular, if they occur with semigroup theory.

This paper focuses on studying the  class of acts with embeddings as an abstract elementary class.  Acts (also known as polygons or $G$-sets) were first introduced in their modern form in the sixties \cite{hoe1}. An $S$-act is a $G$-set where $S$ is only assumed to be a monoid instead of a group.   Their theory has developed into a mature theory as witnessed by \cite{kkm}.

One of the main goals of model theory, and the theory of AECs specifically, is to study dividing lines \cite{sh-div}. A dividing line is a property such that the classes satisfying such a property have some nice behavior while those not satisfying it have a bad one. In this paper, we focus on studying two key dividing lines: stability (see Definition \ref{def-sta}) and superstability (see Definition \ref{def-sup}). They were first introduced in the context of first-order model theory by Shelah in the sixties  \cite{sh1} and later in the context of AECs by Shelah in the nineties \cite{sh394} (see also  \cite{tamenessone} for stability and \cite{grva} for superstability in the context of AECs).

We first  characterize the stability cardinals of the  class of acts with embeddings. The following result mentions the cardinal $\gamma(S)$  which bounds the maximum number of generators of every left ideal of $S$ (see Definition \ref{def-gamma}).

\begin{theorem-2}
Let $\lambda \geq 2^{\operatorname{card}(S) + \aleph_0}$.  The class of  $S\text{-acts}$ with embeddings is stable in $\lambda$ if and only if  $\lambda^{<\gamma(S)} = \lambda$.
\end{theorem-2}

It follows directly from the theorem above that the class of acts with embeddings is \emph{always} stable. For the reader familiar with first-order model theory of acts, observe that this result does not contradict the existence of unstable complete  first-order theories of acts \cite{mus} as types in the context of this paper are essentially quantifier-free types (see Fact \ref{b-types}).

Regarding superstability, we characterize it algebraically via weakly noetherian monoids. Recall that a monoid is weakly left noetherian if every left ideal is finitely generated, see \cite{mil} for more on the algebraic study of these monoids. 

\begin{theorem-3}
Assume $S$ is a monoid. 
$S$ is weakly left noetherian if and only if the class of  $S\text{-acts}$ with embeddings is superstable.
\end{theorem-3}

The theorem above is an extension of \cite[3.12]{maz1} from modules to acts and provides further evidence that superstability is an interesting algebraic notion. It is worth recalling  that there are weakly noetherian monoids for which injective acts are not closed under coproducts (see Remark \ref{lim=ss}). Therefore, Theorem \ref{equiv} puts a bound on how far the arguments given in \cite[\S 5]{mj} can be extended, partially solving the problem stated in the last line of \cite{mj}.

The proof of Theorem \ref{equiv} relies on the algebraic study of limit models  (see Definition \ref{def-limit}). Limit models are universal models with some level of homogeneity. They were introduced by Kolman and Shelah \cite{kosh} and are a key notion that has emerged in the study of AECs  \cite{shvi}, \cite{grvavi}, \cite{grva}, \cite{vasey18}, \cite{limit25}. It is the authors thesis that they are a very interesting algebraic object.

It is known that limit models have a close connection with relative injective objects for modules (see for example \cite{mlim}, \cite{mj}) and in this paper we show that such connection occurs with weakly injective objects for acts (see Section 4.1). Furthermore, we use limit models to obtain an algebraic characterization of parametrized weakly noetherian monoids via parametrized weakly injective acts. The next theorem generalizes \cite[3.7]{g-co} which obtained the result for $\kappa=\aleph_0$.

\begin{theorem-4}
Assume $\kappa$ is a regular cardinal. The following are equivalent.
\begin{enumerate}
\item $S$ is weakly left  $(<\kappa)$-noetherian.
\item  Every $\kappa$-injective act is weakly injective. 
\end{enumerate}
\end{theorem-4}

Finally, we use limit models to obtain a characterization of weakly noetherian monoids via absolutely pure acts. This characterization  extends the classical characterization of noetherian rings via absolutely pure modules obtained by Megibben in the seventies \cite[Theorem 3]{meg}. 

\begin{cor-1}
Assume $S$ is a monoid.
Every absolutely pure $S$-act is weakly injective if and only if $S$ is weakly noetherian.
\end{cor-1}

The proof of the above corollary is not a generalization of the argument for modules: it is a completely different argument which uses in an essential way limit models (see Remark \ref{dif-mod}).

As mentioned at the beginning of this introduction, the objective of the paper was to determine if any meaningful interactions between AECs and algebra could happen outside of module theory. The results of this paper showcase that this is possible by finding interesting interactions and applications of AECs to semigroup theory.

It is worth pointing out that this is not the first model-theoretic study of classes of acts. There is a rich theory developed in the context of first-order model theory for acts, see for example \cite{g-inj}, \cite{mus}, \cite{iva}, \cite{gmps}, \cite{fogo}. The difference between this work and that work is the framework. Although at first AECs can seem convoluted, it is the authors' opinion that they are a more natural set up to study classes of acts in the context of model theory than first-order model theory. We hope that this paper helps convince the reader that this is the case. Furthermore, it is worth noting that Theorem \ref{sta-2} and Theorem \ref{equiv} can be seen as  extensions of two classical results of the first-order model theory of acts \cite{iva}, \cite[3.5]{fogo} (see Remark \ref{first-ext}).


Finally, most of the results we obtain in this paper can be extended from acts to any presheaf category  $\Set^\cs$ where $\cs$ is a small category.\footnote{In the case of S-acts, $\cs$ is the category with one object $*$, one arrow $s: * \to *$ for every $s \in S$ and composition of arrows given by multiplication in $S$.} This is possible by working as in \cite[\S 2]{cfkmr}. We do not work in this general set up in the paper for clarity and since once the set up has been introduced (as in \cite[\S 2]{cfkmr})  the extension is more a booking keeping issue than a mathematical one.

The paper is divided into four sections and an Appendix. Section 2 presents necessary background and basic results. Section 3  characterizes the stability cardinals of the AEC of acts with embeddings. The reader interested in application can skip Subsection 3.2. Section 4 studies limit models in the class of acts with embeddings. We use these to characterize superstability via weakly noetherian monoids and to obtain some new algebraic results. Appendix A contains proofs of some of the main results of the paper using category theory.

 We would like to thank  Daniel Herden and Jonathan Feigert  for discussions around the topics of this paper. We would like to thank Jeremy Beard, Victoria Gould, Wentao Yang and two  anonymous referees for comments that help improve the paper. In particular, we would like to thank Victoria Gould for bringing to our attention \cite[3.5]{fogo} and one of the anonymous referees for Example \ref{ex-ga}.

\section{Preliminaries and basic results}
We present a brief introduction both to acts theory and abstract elementary classes. A more in-depth introduction to acts theory is presented in \cite{kkm}. A more in-depth introduction to abstract elementary classes is presented in \cite[\S \S 4 - 8]{baldwinbook09}. 

\subsection{Acts} Let $S$ be  a monoid with $1 \in S$ as the identity. $A$ is an (left) \emph{S-act} if for every $s\in S$ and $a\in A$ we have $sa\in A$ and $1a=a$ and $s(t(a))=(st)a$ for every $s, t \in S$ and $a \in A$.  We assume that the empty set is an act (which is not allowed in \cite{kkm}). A function $f : A \to B$ is an \emph{S-homomorphism} of acts if $f(sa)=sf(a)$ for every $a \in A$ and $s \in S$. We denote by $S$-$\Act$ the category of S-acts with S-homomorphisms.

 Given $f: A \to B$ a homomorphism, $f$ is an \emph{embedding} or \emph{monomorphism} if $f$ is injective. Furthermore, if the inclusion from $A$ to $B$ is a monomorphism we say that \emph{$A$ is a subact of $B$} and denote it by $A \leq B$. Observe that monomorphisms in $S$-$\Act$ are regular, i.e., they are equalizers of a pair of morphisms.  

S-acts form a variety of unary universal algebras with unary operations $s \cdot $ for every $s \in S$. Hence the category
$S$-$\Act$ is complete and cocomplete. 
Moreover, the forgetful functor $U:S$-$\Act\to\Set$ to the category of sets with functions
preserves limits and colimits.

 In particular,  if $f_1: B \to A_1, f_2: B \to A_2$ are two S-monomorphisms, then the pushout 
$$
		\xymatrix@=2pc{
				B \ar [r]^{f_1}\ar[d]_{f_2} & B_1 \ar[d]^{q_2}\\
			B_2  \ar [r]_{q_1}& P
		}
		$$ 
is the pushout of underlying sets in $\Set$ where the action on $P$ is given by those in $B_1$ and $B_2$. Hence $q_1$ and $q_2$ are monomorphisms. 		
 
Dually, epimorphisms are surjective homomorphisms and they are regular,
i.e., coequalizers of a pair of morphisms. If $f_1: B_1 \to B$ is an epimorphism and
$$
		\xymatrix@=2pc{
				P \ar [r]^{q_2}\ar[d]_{q_1} & B_1 \ar[d]^{f_1}\\
			B_2  \ar [r]_{f_2}& B
		}
		$$ 
a pullback then $q_1$ is an epimorphism. Hence $S$-$\Act$ is a regular category. $S$-$\Act$  is also coregular because pushouts preserve monomorphisms along any homomorphism (and not only along monomorphisms as we stated above).









Given $X \subseteq A$, let $SX = \{ sx : s\in S, x \in X \}$ be the \emph{subact of $A$ generated by $X$}. Observe that in particular if $B$ is a subact of $A$ and $X \subseteq A$, then $S(B \cup X)= B \cup SX$ as there are only unary operations.  An S-act $A$ is \emph{cyclic}
if it is generated by a single element. Equivalently, $A$ is cyclic if there is  an epimorphism from $S$ to $A$.

We say that $I \subseteq S$ is a (left) \emph{ideal} of $S$ if $SI \subseteq I$,
i.e., if $I$ is a subact of $S$.  Recall that given a cardinal $\lambda$, $\lambda^+$  is the least cardinal greater than $\lambda$. 
The next notion will play a key role in this paper.

\begin{defin}\label{def-gamma}
Given a monoid $S$, let $\gamma(S)$ be the minimum infinite cardinal such that every left ideal of $S$ is generated by fewer than $\gamma(S)$ elements. Let $\gamma_r(S)= \gamma(S)$ if $\gamma(S)$ is a regular cardinal and $\gamma(S)^+$ if $\gamma(S)$ is a singular cardinal. 
\end{defin}
 
 We provide an instructive example which was communicated to us by an anonymous referee. 
\begin{example}\label{ex-ga}
Given $\lambda$ an infinite cardinal, let $S_\lambda = \lambda \cup \{ \lambda\}$ and $\alpha \cdot \beta = \operatorname{min} \{ \alpha, \beta \}$. One can show that $(S_\lambda, \cdot)$ is a commutative monoid with identity  $\lambda$. Moreover, every ideal of $S_\lambda$ is of the form $\alpha$ for $\alpha \leq \lambda$. Therefore, if $\lambda$ is a singular cardinal then $\gamma(S_\lambda)= \lambda$; and if $\lambda$ is a regular cardinal then  $\gamma(S_\lambda)= \lambda^+$. 
\end{example}

We will use many concepts from the theory of acts such as weakly injective acts, $\lambda$-injective acts and absolutely pure acts, but we introduce them throughout the text for the convenience of the reader.

\begin{remark}\label{card}
Let $\lambda \geq \operatorname{card}(S)+\aleph_0$. We have that $\lambda^{<\gamma_r(S)}= \lambda$ if and only if $\lambda^{<\gamma(S)}=\lambda$. If $\gamma(S)$ is a regular cardinal there is nothing to show, so we may may assume that $\gamma(S)$ is singular. The forward direction is clear and the backward direction can be obtained as in \cite[5.16]{jec}.
\end{remark}
 
\subsection{Abstract elementary classes} \emph{Abstract elementary classes} (AECs for short) were introduced by Shelah \cite{sh88} to study classes of structures axiomatized in infinitary logics.  An \emph{AEC} is a pair $\Kk=(K, \lea)$ where $K$ is a class of structures in a fixed language and $\lea$ is a partial order on $K$ extending the substructure relation such that $\Kk$ is closed under isomorphisms and satisfies the  coherence property, the L\"{o}wenheim--Skolem--Tarski axiom and the Tarski--Vaught axioms. The L\"{o}wenheim--Skolem--Tarski axiom is an an instance of the Downward L\"{o}wenheim--Skolem theorem and the Tarski--Vaught axioms assure us that the class is closed under directed colimits. The reader can consult the definition in \cite[4.1]{baldwinbook09} or \cite[2.1]{maz1}.

In this paper, we will only work with the \emph{abstract elementary class of (left) S-acts with embeddings} and we will denote it by $\Kk_{\aact} = (S\text{-acts}, \leq)$. In this case (and in the study of AECs of acts), the language will always be $\{ s \cdot   : s \in S \}$ where  $s \cdot$ is interpreted as multiplication by $s$ for each $s \in S$ and the substructure relation is the subact relation.

We will introduce AEC notions in the general set up, but for simplicity one could think that $\Kk$ is always  $\Kk_{\aact}$.

Let $\Kk$ be an AEC. We say that $f: M \to N$ is a \emph{$\Kk$-embedding} if $f: M \cong f[M]$ and $f[M] \lea N$. In the context of $\Kk_{\aact}$, $f$ is a  $\Kk_{\aact}$-embedding if and only if $f$ is a monomorphism of acts. For $A\subseteq M$, we write $f : M \xrightarrow[A]{} N$ if $f: M \to N$ is a $\Kk$-embedding and $f(a) = a$ for every $a \in A$ and we write $f: M \cong_A N$ if $f: M \xrightarrow[A]{} N$ and moreover $f$ is an isomorphism.

We will often write $|M|$ for the underlying set of the model, $\|M\|$ for its cardinality and $\Kk_\lambda$ for the models of $K$ of cardinality $\lambda$. An increasing chain  $\{ M_i : i < \alpha\}\subseteq \Kk$ (for $\alpha$ an ordinal) is a \emph{continuous chain} if $M_i =\bigcup_{j < i} M_j$  for every $i < \alpha$ limit ordinal.

We say that an AEC $\Kk$ has: \emph{amalgamation} if every span $M \lea N_1, N_2$ can be completed to commutative square of $\Kk$-embeddings,  \emph{disjoint amalgamation} if for every  span $M \lea N_1, N_2$ there are $f_1: N_1 \to N$ and $f_2: N_2 \to N$ $\Kk$-embeddings such that $f_1[N_1] \cap f_2[N_2] = f_1[M]$ and $f_1\upharpoonright M =   f_2\upharpoonright M$, \emph{joint embedding} if any two models can be $\Kk$-embedded into a third model, and \emph{no maximal models} if every model can be properly extended. 
\begin{lemma}\label{act-aec}
$\Kk_{\aact} = (S\text{-acts}, \leq)$ is an AEC with $\LS(\Kk_{\aact})=\operatorname{card}(S) + \aleph_0$. Moreover, $\Kk_{\aact}$ has disjoint amalgamation, joint embedding and no maximal models. 
\end{lemma}
\begin{proof}
It is straightforward to show  that $\Kk_{\aact}$ is an AEC with $\LS(\Kk_{\aact})=\operatorname{card}(S) + \aleph_0$. For example,   $\LS(\Kk_{\aact})=\operatorname{card}(S) + \aleph_0$, as given $A$ an S-act and $X \subseteq A$, we  have that $SX$ is a subact of $A$ with $\| SX \| \leq |X| + \operatorname{card}(S) +\aleph_0$. 

Disjoint amalgamation follows from the existence of pushouts and the fact that monomorphisms are stable under pushouts.  Joint embedding and no maximal models follow from the existence of coproducts (which are given by disjoint unions).\end{proof}

We introduce the notion of Galois types which were originally introduced by Shelah. These are a semantic generalization of first-order types.

\begin{defin} Let $\Kk$ be an AEC.
\begin{itemize}
\item The class of \emph{pre-types} $\Kk^{3}$ is the class of triples of the form $(\bar{b}, A,
N)$, where $N \in K$, $A \subseteq |N|$, and $\bar{b} \in N^{<\omega}$.
    \item For $(\bar{b_1}, A_1, N_1), (\bar{b_2}, A_2, N_2) \in \Kk^{3}$, we
say $(\bar{b_1}, A_1, N_1)E_{\text{at}}^{\Kk} (\bar{b_2}, A_2, N_2)$ if $A
:= A_1 = A_2$, and there exist $\Kk$-embeddings $f_\ell : N_\ell \xrightarrow[A]{} N$ for $\ell \in \{ 1, 2\}$ such that
$f_1 (\bar{b_1}) = f_2 (\bar{b_2})$ and $N \in K$. 
\item  Observe that $E_{\text{at}}^{\Kk}$ is a symmetric and
reflexive relation on $\Kk^3$. We let $E^{\Kk}$ be the transitive
closure of $E_{\text{at}}^{\Kk}$.
\item For $(\bar{b}, A, N) \in \Kk^{3}$, \emph{the Galois type of $\bar{b}$ over $A$ in $N$} is the  equivalence class of $(\bar{b}, A, N)$ modulo ${E^{\Kk}}$. We denote it by $\gtp_{\Kk} (\bar{b} / A;
N)$. Usually, $\Kk$ will be clear from the context and we will omit it.
\item For $M \in K$, $\gS_{\Kk}(M)= \{  \gtp_{\Kk}(b / M; N) : M
\leq_{\Kk} N\in K \text{ and } b \in N\} $. Usually, $\Kk$ will be clear from the context and we will omit it.
\end{itemize}
\end{defin} 

\begin{remark} Observe that for every AEC $\Kk$ and $M \in \Kk$ with $\| M \| \geq \LS(\Kk)$, we have that $|\gS_{\Kk}(M)| \leq 2^{\| M \|}$. This is the case as every Galois type contains a triple $(a, M, N)$ with $\| N\|= \|M\|$ and by counting the number of isomorphism types of such triples (see  \cite[2.26]{mt}). So in particular, $\gS_{\Kk}(M)$ is always a set.
\end{remark}

In the context of $\Kk_{\aact}$, we get a simpler characterization of Galois types. 


\begin{remark} Let $M$ be a structure, $X \subseteq M$ and $\bar{b} \in M^{<\omega}$. Recall that \emph{the quantifier-free type of $\bar{b}$ over $X$ in $M$} is the set of (first-order) quantifier-free formulas with parameters in $X$ satisfied by $\bar{b}$ in $M$. We denote it by  $\operatorname{qftp}(\bar{b}/X, M)$.
\end{remark}

\begin{fact}\label{b-types}
Let $X \subseteq A_1, A_2 \in \aact$ and $\bar{b}_1 \in A_1^{<\omega}$, $\bar{b}_2 \in A_2^{<\omega}$. The following are equivalent
\begin{enumerate}
\item  $\gtp(\bar{b}_1/X; A_1)= \gtp(\bar{b}_2/X; A_2)$,
\item $\operatorname{qftp}(\bar{b}_1/X, A_1) = \operatorname{qftp}(\bar{b}_1/X, A_2)$,
\item There is $f: S\bar{b}_1 \cup SX \cong_X S\bar{b}_2 \cup SX$ such that $f(\bar{b}_1)=\bar{b}_2$. 

\end{enumerate}
\end{fact}
\begin{proof}
$(1) \Rightarrow (2)$ is clear and $(3) \Rightarrow (1)$ follows from the fact that $(\bar{b}_1, X,  S\bar{b}_1 \cup SX ) E_{at} (\bar{b}_2, X,  S\bar{b}_2 \cup SX)$. We sketch $(2) \Rightarrow (3)$.

Let $f: S\bar{b}_1 \cup SX \to S\bar{b}_2 \cup SX$ be given by:
 
\[ f(d)=\left\{\begin{array}{cl}s *_{A_2} x & \text{if } d = s *_{A_1} x \text{ for some } x \in X \text{ and }  s \in S \\ s *_{A_2} b_{2,\ell}  & \text{if } d = s *_{A_1} b_{1,\ell} \text{ for some } 1 \leq \ell \leq  n \text{ and } s \in S \end{array}\right.\]

Using $(2)$ it is straightforward to show that $f$ is as required. 
\end{proof}

\begin{remark}
The previous result holds for any universal class in the sense of Tarski \cite{tarski} and it is due to Will Boney \cite[3.7]{vaseyd}. In particular,  the result holds for every variety, in the sense of universal algebra, with embeddings. 
\end{remark}

\begin{defin} Let $\Kk$ be an AEC and $\lambda > \LS(\Kk)$. 
\begin{itemize}
\item $M $ is \emph{$\lambda$-saturated} if   for all $M_0 \lea M$ with $\LS(\Kk) \leq \|M_0\| < \lambda$, and all $p \in \gS(M_0)$, there is $m \in M$ such that $p = \gtp(m/M_0; M)$.
\item $M$ is \emph{saturated} if $M \in \Kk_\lambda$ and $M$ is $\lambda$-saturated.
\end{itemize}  
  \end{defin}
  
  \begin{fact}\label{sat-iso}
Let $\Kk$ be an AEC with joint embedding, amalgamation and no maximal models. If $M$ and $N$ are saturated models and $\| M \| = \| N \| > \LS(\Kk)$, then $M$ and $N$ are isomorphic. 
\end{fact}

One of the key model-theoretic dividing lines is stability. We will study stability in detail in Section 3.

\begin{defin}\label{def-sta} Let $\Kk$ be an AEC.
\begin{itemize}
\item Assume $\lambda \geq \LS(\Kk)$. An AEC $\Kk$ is \emph{stable in $\lambda$}  if for every $M \in
\Kk_\lambda$, $| \gS_{\Kk}(M) | \leq \lambda$.  
\item An AEC $\Kk$ is \emph{stable} if there is a $\lambda \geq \LS(\Kk)$ such that $\Kk$ is stable in $\lambda$.
\end{itemize}
\end{defin}

\begin{defin}  Assume $\lambda \geq \LS(\Kk)$.
\emph{$M$ is universal over $N$} if and only if $N \lea M$, $\| M  \| = \| N \| = \lambda$
 and for any $N^* \in \Kk_{\lambda}$ such that
$N \lea N^*$, there is a $\Kk$-embedding $f: N^* \xrightarrow[N]{} M$.
\end{defin} 

The existence of universal extensions follows from stability. 

\begin{fact}[{\cite[2.9]{tamenessone}}]\label{uni-exis}
Let $\Kk$ be an AEC with  amalgamation and no maximal models. If $\Kk$ is stable in $\lambda$, then for every  $M
\in \Kk_\lambda$, there is $N \in \Kk_\lambda$ such that $N$ is universal over $M$.
\end{fact}

 We recall limit models which were originally introduced in \cite{kosh}.  They will be studied in detail in Section 4 in $\Kk_{\aact}$. We will show that they are very interesting algebraic objects (see Section 4.2 for some applications to algebra).  

\begin{defin}\label{def-limit}
Let $\lambda  \geq \LS(\Kk)$  and $\alpha < \lambda^+$ be a limit ordinal.  \emph{$M$ is a $(\lambda,
\alpha)$-limit model over} $N$ if and only if there is $\{ M_i : i <
\alpha\}\subseteq \Kk_\lambda$ an increasing continuous chain such
that:
\begin{itemize}
\item $M_0 =N$ and  $M= \bigcup_{i < \alpha} M_i$, and
\item $M_{i+1}$ is universal over $M_i$ for each $i <
\alpha$.
\end{itemize}

$M$ is a \emph{$(\lambda, \alpha)$-limit model} if there is $N \in
\Kk_\lambda$ such that $M$ is a $(\lambda, \alpha)$-limit model over
$N$. $M$ is a \emph{$\lambda$-limit model} if there is a limit ordinal
$\alpha < \lambda^+$ such that $M$  is a $(\lambda,
\alpha)$-limit model.  $M$ is a \emph{limit model} if there is an infinite cardinal $\lambda \geq \LS(\Kk)$ such that $M$ is a $\lambda$-limit model.
\end{defin}

\begin{remark}\label{e-limit}  Assume $\Kk$ is an AEC with joint embedding, amalgamation and no maximal models, and $\Kk$ is stable in $\lambda$. It follows from Fact \ref{uni-exis} that for every $M \in \Kk_\lambda$ and $\alpha < \lambda^+$ limit ordinal there is $N$ a $(\lambda, \alpha)$-limit model over $M$. Moreover, a back-and-forth argument shows that for any $\alpha < \lambda^+$ limit ordinal there is a unique 
$(\lambda, \alpha)$-limit model.
\end{remark}

We introduce superstability using limit models. Superstability is also an important dividing line in model theory. We will study superstability in detail in Section 4.

\begin{defin}\label{def-sup}
We say that $\Kk$ is \emph{superstable} if and only if there is a cardinal $\mu \geq \LS(\Kk)$ such that $\Kk$ has a unique $\lambda$-limit model (up to isomorphisms) for every $\lambda \geq \mu$. 
\end{defin}

\begin{remark}\label{limit1}
Assume $\Kk$ is an AEC with joint embedding, amalgamation and no maximal models. It is easy to show that if $\Kk$ has a $\lambda$-limit model then $\Kk$ is stable in $\lambda$. Therefore, if $\Kk$ is superstable, then $\Kk$ is stable. 
\end{remark}

\begin{remark}\label{tail} 
For the reader familiar with  superstability in the context of first-order model theory, Definition \ref{def-sup} is equivalent to being stable on a tail of cardinals for AECs with amalgamation, joint embedding, no maximal models and tameness  \cite[1.3]{grva}, \cite{vaseyt}. In particular, the equivalence holds for  first-order theories and $\Kk_{\aact}$ (see Remark \ref{lim=ss}). 

\end{remark}

\section{Stability spectrum}
 
 The main result of this section is a characterization of the stability cardinals of $\Kk_{\aact}$. 

\begin{theorem}\label{sta-2}
Let $\lambda \geq 2^{\operatorname{card}(S) + \aleph_0}$.  $ (S\text{-acts}, \leq)$ is stable in $\lambda$ if and only if  $\lambda^{<\gamma(S)} = \lambda$.
\end{theorem}

\subsection{Stability} Independence relations on categories were introduced in \cite{lrv1}. These extend Shelah's notion of non-forking which in turn extends linear independence. In this subsection we study a specific independence relation on $\Kk_{\aact}$ while in parallel quickly introducing stable independence relation in the context of AECs. All the abstract notions we introduce in this subsection were first introduced in \cite{lrv1}. Furthermore we will always work in $\Kk_{\aact}$  unless specified otherwise. 


\begin{defin}[{\cite[3.4]{lrv1}}] An \emph{independence relation} on an AEC $\Kk$ is a collection $\dnf$ of commutative squares of $\Kk$-embeddings such that for any commutative diagram:

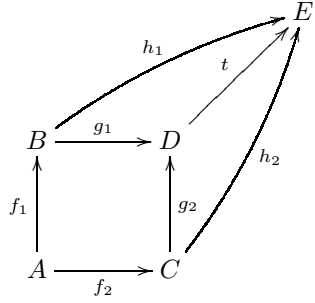
\[
  \xymatrix@=3pc{
    & & E \\
    B \ar[r]^{g_1}\ar@/^/[rru]^{h_1} & D \ar[ru]^t  & \\
    A \ar [u]^{f_1} \ar[r]_{f_2} & C \ar[u]_{g_2} \ar@/_/[ruu]_{h_2} &
  }
\]

we have that $(f_1, f_2, g_1, g_2) \in \dnf$ if and only if $(f_1, f_2, h_1, h_2) \in \dnf$.

\end{defin} 

We study the following relation on $\Kk_{\aact}$.  
\begin{defin}\label{indp} 
We say that $(f_1, f_2, g_1, g_2) \in  \dnfs$ if and only if all the maps are S-monomorphisms and the following diagram

\[
 \xymatrix{\ar @{} [dr]  A_1 \ar[r]^{g_1}  & B \\
A_0 \ar[u]^{f_1} \ar[r]^{f_2} &  A_2  \ar[u]_{g_2} 
 }
\]
is a pullback. This means that $g_1[A_1] \cap g_2[A_2] = g_1 \circ f_1 [A_0] (=g_2 \circ f_2 [A_0])$.
\end{defin}

\begin{remark}
It is straightforward to show that $\dnfs$ is an independence relation on $\Kk_{\aact}$. 
\end{remark}

\begin{remark}
Notice that we use the symbol $\dnf$ for an arbitary independence relation on an AEC $\Kk$ and the symbol $\dnfs$ for the independence relation on $\Kk_{\aact}$ defined on Definition \ref{indp}.
\end{remark}

Given $\Kk$ an AEC. An independence relation $\dnf$ on $\Kk$ is \emph{weakly stable} if it satisfies symmetry \cite[3.9]{lrv1}, existence \cite[3.10]{lrv1}, uniqueness \cite[3.13]{lrv1} and transitivity \cite[3.13]{lrv1}.

\begin{defin}[{\cite[3.24, 8.14]{lrv1}}]\label{def-ind} Let $\Kk$ be an AEC. We say that $\dnf$ is a \emph{stable independence relation}  on $\Kk$ if $\dnf$ is weakly stable and satisfies the witness property \cite[8.7]{lrv1} and local character \cite[8.6]{lrv1}.
\end{defin}

\begin{remark}\label{e-squares} The category $S$-$\Act$ has \textit{effective unions} in the sense of Barr
\cite{B}. This means that whenever we have a pullback of $S$-monomorphisms
$$
\xymatrix@=3pc{
A_1 \ar[r]^{g_1} & A_3 \\
A_0 \ar [u]^{f_1} \ar [r]_{f_2} &
A_2 \ar[u]_{g_2}
}
$$
and the pushout
$$
\xymatrix@=3pc{
A_1 \ar[r]^{q_1} & P \\
A_0 \ar [u]^{f_1} \ar [r]_{f_2} &
A_2 \ar[u]_{q_2}
}
$$
then the induced morphism $k:P\to A_3$ is an $S$-monomorphism. This immediately
follows from the fact that the category $\Set$ has this property and the forgetful functor $U:S$-$\Act\to\Set$ preserves limits and colimits. It is also easy to check that the induced morphism is a monomorphism using an explicit construction of the pushout in $S$-$\Act$ (see for example \cite[\S II.2.26]{kkm}).

It is worth pointing out that this result will only be used in Fact \ref{ind} and in the Appendix.
\end{remark}


\begin{fact}\label{ind}
Let $\dnfs$ be as in Definition \ref{indp}. $\dnfs$ is a stable independence relation on $\Kk_{\aact}$.
\end{fact}
\begin{proof} The  result follows from Remark \ref{e-squares}, the fact that $\aact$ is coregular and  \cite[5.1]{lrv1}. 
\end{proof}

\begin{remark}
Fact \ref{ind} also has a short proof using an explicit construction of the pushout in $S$-$\Act$. 
\end{remark}

We now proceed to do a detailed analysis of the  independence relation introduced in Definition \ref{indp}.

\begin{nota} Let $\Kk$ be an AEC with an independence relation $\dnf$.

We write $M_1 \dnf^{M_3}_{M_0} M_2$ if $M_0 \lea M_1, M_2 \lea M_3$ and $(i_{0,1}, i_{0,2},i_{1,3},i_{2,3}) \in \dnf$ where $i_{\ell, m}$ is the inclusion map for every $\ell, m$. \end{nota}

On AECs, the independence relation can be extended to sets.

\begin{defin}[{\cite[8.2]{lrv1}}] Let $\Kk$ be an AEC with an independence relation $\dnf$. 
\emph{$X$ is (bar-)free from $Y$ over $N_0$ in $N_3$}, denoted by $\dnfb{N_0}{X}{Y}{N_3}$, if $N_0 \lea N_3$, $X \cup Y \subseteq |N_3|$ and there are $M_1, M_2, M_3 \in \Kk$ such that $X \subseteq |M_1|$, $Y \subseteq |M_2|$, $N_3 \lea M_3$ and $M_1 \dnf^{M_3}_{N_0} M_2$.

\end{defin}

We characterize the previous notion in $\Kk_{\aact}$.

\begin{lemma}\label{bar-nf} Let $A,B$ be S-acts, $A  \leq B$ and $X, Y \subseteq B$. Then  $\dnfbs{A}{X}{Y}{B}$ if and only if $SX \cap SY \subseteq A$. 
\end{lemma}
\begin{proof} The forward direction is clear so we prove the backward direction.  Assume $SX \cap SY \subseteq A$. Let $L_1 = SX \cup A$,  $L_2 = SY \cup A$ and $L_3 = B$. It is clear that $L_1, L_2 \leq  L_3$ are such that $X \subseteq |L_1|$, $Y \subseteq |L_2|$, $B \leq L_3$ and since $SX \cap SY \subseteq A$  we have that $L_1 \cap L_2 \subseteq A$. Hence $L_1 \dnfs^{L_3}_{A} L_2$. Therefore, $\dnfbs{A}{X}{Y}{B}$.
\end{proof}

We extend the independence relation to Galois types.

\begin{defin} Let $\Kk$ be an AEC with an independence relation $\dnf$. 

Let $\bar{a} \in N^{< \omega}$, $X \subseteq |N|$ and $M \lea N$.
\emph{$\gtp(\bar{a}/X, N)$ does not $\dnf$-fork over $M$}
 if $|M| \subseteq X$ and $ \dnfb{M}{ran(\bar{a}) }{X}{N}$ where $ran(\bar{a})$ denotes the range of $\bar{a}$. 
\end{defin}

\begin{remark}
It is straightforward to show that non-forking is well-defined for types, i.e., if $(\bar{a}_1, X, N_1) E_{\text{at}} (\bar{a}_2, X, N_2)$, $|M| \subseteq X$ and $ \dnfb{M}{ran(\bar{a}_1) }{X}{N_1}$ then $ \dnfb{M}{ran(\bar{a}_2) }{X}{N_2}$.
\end{remark}

The following result holds for any stable independence relation \cite[8.5]{lrv1}, but we provide the proof in the setting of this paper as the argument is significantly simpler than in the general case and it helps to elucidate the setting of the paper. 

\begin{fact}\label{uniq}(Uniqueness) Let $C$ be an S-act. Let $p, q \in \gS_{\Kk_{\aact}}(C)$ with $B \leq C$. If $p\rest B = q\rest B$ and $p, q$ do not $\dnfs$-fork over $B$, then $p=q$. 
\end{fact}
\begin{proof}
Let $p = \gtp(a_1 /C, D)$ and $q = \gtp(a_2/C, D)$, we may assume they are both realized in a fixed $D$ by the amalgamation property. Since $p\rest B = q\rest B$ there is  $f: Sa_1 \cup B \cong_B Sa_2 \cup B$ such that $f(a_1)=a_2$ by Fact \ref{b-types}.

Let $g: Sa_1 \cup C  \rightarrow  Sa_2 \cup C$ be given by:
 
\[ g(d)=\left\{\begin{array}{cl}f(d) & d \in Sa_1 \\ d &  d\in C\end{array}\right.\]
Observe that $g$ is a function because $f\rest {Sa_1 \cap C} = 1_{Sa_1 \cap C}$ as $C \cap Sa_1 \subseteq B$  since $p$ does not $\dnfs$-fork over $B$.  As we only have unary functions and  $f$ is a homomorphism, clearly $g$ is a homomorphism. It is clear that  $g$ is surjective so we show $g$ is injective. The only interesting case to consider is when $d \in Sa_1$, $c \in C$ and $f(d)=g(c)$. Since  $g(c) =c $ and $f(d) \in Sa_2$ because $f(a_1)=a_2$, it follows that $c\in C \cap Sa_2$. As $C \cap Sa_2 \subseteq B$ because $q$ does not $\dnfs$-fork over $B$, we have that $f(c)=c$. Hence $f(d) = f(c)$.  As $f$ is injective, $c=d$.  

Therefore, $\gtp(a_1 /C, D) = q = \gtp(a_2/C, D)$. \end{proof}

\begin{remark}
The previous result is similar to \cite[4.1]{mus}.
\end{remark}

\begin{lemma}\label{small-nf} Let $A$ be a subact of $B$. If $x \in B \backslash A$  and $A \cap Sx \neq \emptyset$, then there  is $Z \subseteq A$ such that $|Z| < \gamma_r(S)$ and $Sx \dnfs^{B}_{SZ} A$.
\end{lemma}
\begin{proof}
Let $Y = A \cap Sx $. We show that there is $Z \subseteq Y$ such that $| Z | < \gamma_r(S)$ and $SZ=Y$. Observe that this is clearly  enough as such a $Z$ would satisfy that $Sx \dnfs^{B}_{SZ} A$. 

Assume for the sake of contradiction that such a $Z$ does not exist. Then one can build $\{ z_\alpha : \alpha < \gamma_r(S) \} \subseteq Y$ such that $z_\alpha \notin S \{ z_\beta : \beta < \alpha\}$ by induction on $\alpha < \gamma_r(S)$. As every $z_\alpha \in Sx$, for every $\alpha < \gamma_r(S)$ there is $t_\alpha \in S$ such that $z_\alpha = t_\alpha x$. 

Let $I = S\{t_\alpha : \alpha < \gamma_r(S) \}$. Then there is $J \subseteq \gamma_r(S)$ such that $J\neq \emptyset$, $| J| < \gamma(S)$ and $I = S\{ t_j : j \in J\}$ by the definition of $\gamma(S)$. Since $\gamma_r(S)$ is regular and $|J| < \gamma(S)$, there is $\alpha < \gamma_r(S)$ such that $j <  \alpha$ for every $j \in J$. Hence there is $j_0 \in J$ such that $j_0 < \alpha$ and $t_\alpha = s t_{j_0}$ for some $s \in S$ . 

Therefore, $z_\alpha = t_\alpha x =  s t_{j_0} x = sz_{j_0}$ where the first and last equality follow from the last equation of the second paragraph of the proof and the middle equality follows from the last equation of the previous paragraph. Hence $z_\alpha =  sz_{j_0} \in Sz_{j_0}  \subseteq S \{ z_\beta : \beta < \alpha\}$. This contradicts the choice of $z_\alpha$. 
\end{proof}

The following result follows directly from Lemma \ref{small-nf}.

\begin{cor}\label{nf-gtp}
If $p = \gtp(x/B, N) \in \gS_{\Kk_{\aact}}(B)$, then there is a subset $Z$ of $B$ such that $|Z| < \gamma_r(S)$ and $p$ does not $\dnfs$-fork over $SZ$, i.e., $\dnfbs{SZ}{x}{B}{N}$.
\end{cor}

We obtain the first half of the main result of this section.

\begin{lemma}\label{sta} Let $\lambda \geq 2^{\operatorname{card} (S)+ \aleph_0}$.
If $\lambda^{<\gamma(S)} = \lambda$, then $\Kk_{\aact}$ is stable in $\lambda$.
\end{lemma}
\begin{proof}
Let $B$ be an act of cardinality $\lambda$. Assume for the sake of contradiction that $| \gS(B)| > \lambda$. Let $\{ p_i :  i < \lambda^+ \} \subseteq \gS(B)$ such that $p_i \neq p_j$ if $i \neq j < \lambda^+$.

Let $\Phi: \lambda^+ \to  \mathcal{P}^{<\gamma_r(S) }(B) = \{ X \subseteq B : |X| < \gamma_r(S)\}$ be given  by $\Phi(i)= B_i$ such that $p_i$ does not $\dnfs$-fork over $SB_i$ and $B_i\in \mathcal{P}^{<\gamma_r(S)  }(B)$. Observe that the $B_i$ exists by Corollary \ref{nf-gtp}  and that $|\mathcal{P}^{<\gamma_r(S)  }(B)| \leq \lambda$ because $\lambda^{<\gamma(S)} = \lambda$ and by Remark \ref{card}. Then by the pigeon hole principle there are $X \subseteq \lambda^+$ and $B^* \in \mathcal{P}^{< \gamma_r(S) }(B)$ such that $|X| = \lambda^+$ and for every $i \in X$, $p_i$ does not $\dnfs$-fork over $SB^*$. 

Let $\Psi:  X \to \gS(SB^*)$ given  by $\Psi(i)= p_i \rest SB^*$. Observe that   $| \gS(S B^*) | \leq 2^{\| S B^* \|  + \aleph_0}  \leq 2^{\operatorname{card} (S)+ \aleph_0} \leq \lambda$. Then by applying the pigeon hole principle again there are $i \neq j \in X$ such that $p_i \rest SB^* = p_j \rest SB^*$. Hence $p_i = p_j$ by Fact \ref{uniq}. This is a contradiction as $p_i \neq p_j$ by construction. \end{proof}

\begin{remark}
The argument to get stability in $\lambda$ works for any AEC with intersection that has a stable independence relation such that that there is a cardinal $\kappa$ such that  for every $p = \gtp(x/B, N) \in \gS(B)$ there is a subset $Z$ of $B$ with $|Z | < \kappa$ and $p$ does not fork over $cl(Z)$. A similar argument is presented in \cite[7.5, 7.15]{gm}.
\end{remark}

\begin{remark}
Lemma \ref{sta} in particular shows that $\Kk_{\aact}$ is always a stable AEC. This result does not contradict the fact that there are unstable complete first-order  theories of acts \cite{mus} as we are basically only considering quantifier-free types (see Remark \ref{b-types}).
\end{remark}

We finish this subsection by presenting an algebraic result  which follows from Lemma \ref{small-nf}.  

\begin{cor}\label{sub-cyc}
If $A$ is a subact of $B$ and $B$ is cyclic, then $A$ is generated by fewer than $\gamma_r(S)$ elements.
\end{cor}
\begin{proof}  Assume $B =Sx$ for some $x \in B$ and $A \neq \emptyset$. If $x \in A$ we are done, so assume $x \notin A$. As $x \in B \backslash A$ and $A \cap Sx = A \neq \emptyset$, it follows from Lemma \ref{small-nf} that there is  $Z \subseteq A$ such that $|Z| < \gamma_r(S)$ and $Sx \dnfs^{B}_{SZ} A$. Then $A = A \cap Sx \subseteq SZ$. Hence $Z$ is as required.
 \end{proof}

 The reader interested in the algebraic characterization of superstability or the algebraic results of the next section can skip the rest of this section.

\subsection{Characterizing the stability cardinals} In this subsection we characterize all the stability cardinals  of $\Kk_{\aact}$. In order to do that, we use some deep results on AECs \cite{vaseyt}.

Assume $\Kk$ is an AEC and let $M \in \Kk_\lambda$, $M \lea N$ and  $p \in \gS(N)$. Recall that  \emph{$p$ ($\lambda$-)splits over $M$} if there are $N_1, N_2 \in \Kk_\lambda$ and $h: N_1 \cong_ M N_2$ such that $M \lea N_1, N_2 \lea N$ and $h(p\rest N_1) \neq p\rest N_2$ \cite[Definition 3.2]{sh394}.

\begin{lemma}\label{split-dnf}   Let $\lambda \geq \operatorname{card} (S)+ \aleph_0$. Assume $L_1 \leq L_2 \leq M$ are all  $S$-acts of cardinality $\lambda$ and $L_2$ is universal over $L_1$.
If $p \in \gS_{\Kk_{\aact}}(M)$ does not split over $L_1$, then $p$ does not $\dnfs$-fork over $L_2$.
\end{lemma}
\begin{proof}
Let  $p = \gtp(a/M, N)$ for $a \in N$ and $M\leq N \in (\Kk_{\aact}
)_\lambda$. Assume for the sake of contradiction that $p$ $\dnfs$-forks over $L_2$. Then $Sa \cap M \nsubseteq L_2$ by Lemma \ref{bar-nf}, let $b \in (Sa \cap M)  \backslash L_2$.

Since $L_2$ is universal over $L_1$, there is $f: N \xrightarrow[L_1]{} L_2$. Let $N_1 = Sb \cup L_1$ and $N_2 = f[N_1]$. We show that $N_1, N_2$ witness that $p$ splits over $L_1$.

It is clear that $f:N_1 \cong_{L_1} N_2$ and $L_1 \leq N_1, N_2 \leq M$ as $b \in M$. We show that $f(p\upharpoonright N_1) \neq p \upharpoonright N_2$. Assume for the sake of  contradiction that $f(p\upharpoonright N_1)  =  p \upharpoonright N_2$, then $\operatorname{qftp}(f(a)/f[N_1], f[N]) = \operatorname{qftp}(a/f[N_1], N)$ by Fact \ref{b-types}. Observe that there is $s_1 \in S$ such that $s_1x= f(b) \in \operatorname{qftp}(f(a)/f[N_1], f[N]) $ and $N \models b = s_1 a$ as $b \in Sa$. As $\operatorname{qftp}(f(a)/f[N_1], f[N]) = \operatorname{qftp}(a/f[N_1], N)$,  $N \models s_1a =f(b)$. This is a contradiction as   $b = s_1a = f(b) \in L_2$ but $b \notin L_2$ by the the choice of $b$. 
\end{proof}

\begin{remark}
Shortly after obtaining Lemma \ref{split-dnf}, the first author surprisingly showed that the result holds in many AECs \cite[4.7]{limit25}. 
\end{remark}


The following notion was introduced in \cite[2.4, 3.8]{vaseyt} in a more general set-up.\footnote{It is equivalent to the definition of \cite{vaseyt} if $\Kk$ has amalgamation, joint embedding, no maximal models and $(<\aleph_0)$-tameness. All of these properties hold in $\Kk_{\aact}$ by Lemma \ref{act-aec} and Fact \ref{b-types}.}

\begin{defin}
Let $\Kk$ be an AEC stable in $\lambda$. 
\begin{enumerate}
\item Let $\underline{\kappa}(\Kk_\lambda, \lea^{u})$ be the set of ordinals $\delta < \lambda^+$ such that if  $\{ M_i : i \leq \delta \} \subseteq K_\lambda$  is an increasing continuous chain with $M_{i+1}$  universal over $M_i$ for every $i < \delta$ and  $p \in \gS(M_\delta), \text{ then there is } i <\delta \text{ such that } p \text{ does not split over } M_i$.

\item Let $\kappa(\Kk_\lambda, \lea^{u})$ be the least regular cardinal in $\underline{\kappa}(\Kk_\lambda, \lea^{u})$ if such a cardinal exists.
\end{enumerate}
\end{defin}

\begin{remark} If $A\leq B, C$ and $B \cap C \subseteq A$, then  $B \cup C$ is an $S$-act where 

$$s *_{B \cup C}  d=\left\{\begin{array}{cl} s *_B d & d \in B \\ s *_C d & d \in C    \end{array}\right.$$
such that $B, C \leq B \cup C$.

We denote this act by $B \cup C$.
\end{remark}

\begin{lemma}\label{g=k}
If $\Kk_{\aact}$ is stable in $\lambda$ and $\lambda \geq (\operatorname{card}(S) + \aleph_0)^+$, then $$\kappa((\Kk_{\aact})_\lambda, \leq^u) = \gamma_r(S).$$
\end{lemma}
\begin{proof}
We show first that $\gamma_r(S) \in \underline{\kappa}((\Kk_{\aact})_\lambda, \leq^u)$.  Let $\{ M_i : i \leq  \gamma_r(S) \} \subseteq (\Kk_{\aact})_\lambda$ be a continuous chain such that $M_{i+1}$ is universal over $M_i$ for every $i < \gamma_r(S)$ and  $p \in \gS(M_{\gamma_r(S)})$. Then there is $Z \subseteq M_{\gamma_r(S)}$ such that $|Z| < \gamma_r(S)$ and $p$ does not $\dnfs$-fork over $SZ$ by Lemma \ref{nf-gtp}. So there is $i < \gamma_r(S)$ such that $p$ does not $\dnfs$-fork over $M_i$ as $|Z| < \gamma_r(S)$ and $\gamma_r(S)$ is regular. Therefore, $p$ does not split over $M_i$ by \cite[4.2]{bgkv}. Hence $\gamma_r(S) \in \underline{\kappa}((\Kk_{\aact})_\lambda, \leq^u)$. 

As $\gamma_r(S) \in \underline{\kappa}((\Kk_{\aact})_\lambda, \leq^u)$, it follows that  $\kappa((\Kk_{\aact})_\lambda, \leq^u) \leq \gamma_r(S)$ by minimality of $\kappa((\Kk_{\aact})_\lambda, \leq^u)$. Assume for the sake of contradiction that ${\kappa((\Kk_{\aact})_\lambda, \leq^u)} < \gamma_r(S)$ and denote for simplicity $\kappa((\Kk_{\aact})_\lambda, \leq^u)$ by $\kappa$. Then there is an ideal $I$ which is not $(<\kappa)$-generated by the definition of $\gamma_r(S)$. Hence there are $\{ a_i : i < \kappa \} \subseteq S$ such that for every $i <\kappa$, $a_{i} \notin S \{ a_j : j < i \}$.

We build $\{ B_i : i < \kappa \} \subseteq (\Kk_{\aact})_\lambda$ an increasing continuous chain by induction such that:

\begin{enumerate}
\item for every $i< \kappa$, $B_{i+1} $ is universal over $B_{i}$, 

\item for every $i < \kappa$, $S\{a_j : j < i\} \leq B_{i}$, and
\item  for every $i< \kappa$, $B_{i} \cap S = S\{a_j : j <  i \}$ and $B_0 \cap S = \emptyset$. 
\end{enumerate}

\fbox{Enough} Let $B_{ \kappa}  = \bigcup_{i<   \kappa} B_i$. Observe that $B_{ \kappa} \cap S  = S\{ a_i : i < \kappa\}$ and $1 \notin B_{ \kappa}$  by Condition (3) of the construction.  Let $p = \gtp(1/  B_{\kappa}, B_{\kappa} \cup S ) \in \gS(B_{\kappa})$. Then there is an $i < \kappa$ such that $p$ does not split over $B_{i} $ by Condition (1) of the construction and the definition of $\kappa((\Kk_{\aact})_\lambda, \leq^u)$. Therefore,  $p$ does not $\dnfs$-fork over $B_{i + 1}$  by Condition (1) of the construction and Lemma \ref{split-dnf}. Hence $\dnfbs{ B_{i + 1}  }{1}{B_\kappa}{B_\kappa \cup S}$.

So $S1 \cap B_\kappa  \subseteq  B_{i + 1} $ by Lemma \ref{bar-nf}. It follows from Condition (3) of the construction that $S1 \cap B_\kappa =   S\{ a_j : j < \kappa\}$ and that $S \cap B_{i + 1} =  S \{a_j : j < i +1  \} $. Hence $  S\{ a_j : j <\kappa \} \subseteq S \{a_j : j < i +1  \} $. This is a contradiction as $a_{i+1} \notin   S \{a_j : j < i +1  \} $ by the choice of the $a_i$'s. 

\fbox{Construction} In the base step, let $B_0$ be an S-act of cardinality $\lambda$ such that $B_0 \cap S = \emptyset$. We carry out the induction step. If $j$ is a limit ordinal take unions, so we are left with the case when $j = i +1$. Since $ \Kk_{\aact}$ is stable in $\lambda$ and $B_i  \cup S\{a_j : j < i +1 \}  \in (\Kk_{\aact})_\lambda$ as $B_i \cap   S \{a_j : j  <  i +1 \}=  S\{a_j : j < i \} \leq B_{i}, S\{a_j : j < i +1 \}$, there is $C \in (\Kk_{\aact})_\lambda$ universal over  $B_i  \cup S\{a_j : j < i +1 \}  $. Then there is $B_{i+1}$ and $f$ an isomorphism such that  the following diagram commutes:

  \[
 \xymatrix{\ar @{} [dr]  B_{i+1} \ar[r]^{f}  & C \\
B_i  \cup S\{a_j : j < i +1 \}  \ar[u]^{\leq} \ar[r]^{\id} &  B_i  \cup S\{a_j : j < i +1 \} \ar[u]_{\leq} 
 }
\]

and $B_{i+1} \cap S = ( B_i  \cup S\{a_j : j < i +1 \}) \cap S$. 

It is clear that $B_{i+1}$ satisfies Conditions (2) and (3) and it satisfies Condition (1) by the amalgamation property and the choice of $C$. 
 \end{proof}

\begin{cor}\label{o-dir}
If $\lambda \geq 2^{\operatorname{card}(S) + \aleph_0}$ and $\Kk_{\aact}$ is stable in $\lambda$, then $\lambda^{<\gamma(S)} = \lambda$.
\end{cor}
\begin{proof}
It follows from \cite[10.8]{vaseyt} that $\lambda^{<\kappa^{wk}((\Kk_{\aact})_{\lambda^{*}}, \leq^u)} = \lambda$ for some  cardinal $\lambda^*  \geq (2^{\operatorname{card}(S) + \aleph_0})^+ $ such that $\Kk_{\aact}$ is stable in $\lambda^*$.\footnote{The cardinal 
${\kappa^{wk}((\Kk_{\aact})_{\lambda^*}, \leq^u)}$ is the least regular cardinal in $\underline{\kappa}^{\operatorname{wk}}((\Kk_{\aact})_{\lambda^*}, \lea^{u})$.  $\underline{\kappa}^{\operatorname{wk}}((\Kk_{\aact})_{\lambda^*} , \lea^{u})$ is defined as $\underline{\kappa}^{\operatorname{wk}}((\Kk_{\aact})_{\lambda^*} , \lea^{u})$ with the exception that one changes ``$p$ does not split over $M_i$" by ``$p\rest M_{i+1}$ does not split over $M_i$". These were originally introduced in \cite[2.4, 3.8]{vaseyt} for tame AECs.}  Observe that ${\kappa((\Kk_{\aact})_{\lambda^*}, \leq^u)} = {\kappa^{wk}((\Kk_{\aact})_{\lambda^*}, \leq^u)}$ since $\Kk_{\aact}$ is $(<\aleph_0)$-tame by Fact \ref{b-types} and by \cite[3.7, 2.8]{vaseyt} (see \cite[2.39]{limit25} for additional details). Hence $\lambda^{<\kappa((\Kk_{\aact})_{\lambda^*}, \leq^u)} = \lambda$. The result then follows from Lemma \ref{g=k} and Remark \ref{card}.
\end{proof}

We obtain the main result of this section. The result is stated at the beginning of the section.

\begin{proof}[Proof of Theorem 3.1]
Follows directly from Lemma \ref{sta} and Corollary \ref{o-dir}.
\end{proof}

The previous result is analogous to the result for modules \cite[3.6]{maz1} with the expection that in the result for modules $\lambda \geq (\operatorname{card}(S) + \aleph_0)^+$ instead of $ 2^{\operatorname{card}(S) + \aleph_0}$. It is worth pointing out that the proof of the result for modules is based on determining the cardinality of the injective envelope \cite{eklof}, while the result we present in this paper is proven by a purely model theoretic argument. 

\begin{question} \
\begin{enumerate}
\item Can  $2^{\operatorname{card}(S) + \aleph_0}$ be replaced by $(\operatorname{card}(S) + \aleph_0)^+$ in Theorem \ref{sta-2}?
\item Provide a proof of Corollary \ref{o-dir} using only algebraic methods. 
\end{enumerate}
\end{question}

\section{Limit models}

This section focuses on studying limit models (see Definition \ref{def-limit}) in $\Kk_{\aact}$. The first half of the section characterizes superstability algebraically while the second half of the section has some algebraic results that use limit models. We will always work in $\Kk_{\aact}$.

\subsection{Superstability}

 Recall that an act $A$ is \emph{$\kappa$-injective} if for every $I \subseteq S$ an ideal generated by fewer than $\kappa$ elements and $f: I \to A$ an $S$-homomorphism there is an $S$-homomorphism $g: S \to A$ such that $f(d) = g(d)$ for every $d\in I$ \cite[\S 2]{g-co}. It is important to note that a $\gamma(S)$-injective act might not be injective \cite[\S III.5]{kkm}. $\gamma(S)$-injective acts are called \emph{weakly injective acts}.

\begin{lemma}\label{l-inj} Assume $\delta < \lambda^+$ is a limit ordinal.
If $A$ is a $(\lambda, \delta)$-limit model, then $A$ is $\text{cf}(\delta)$-injective. Moreover, if $\text{cf}(\delta) \geq \gamma_r(S)$, then $A$ is injective. 
\end{lemma}
\begin{proof} The argument given in \cite[3.2]{mlim} can be carried out in our setting (see also Lemma \ref{o-abs} for an idea of the proof). The moreover part follows from Corollary \ref{sub-cyc} and Skornjakov--Baer criterion for injectivity (see for example \cite[\S III.1.8]{kkm}).\end{proof}

In the case of modules,  Bumby's classical result for injective modules \cite{bumby} that  if $f: A \to B$, $g: B \to A$ are embeddings and $A, B $ are injective modules, then $A$ is isomorphic to $B$, can be used to show that long limit models are isomorphic. Since we do not know if Bumby's result holds for acts we have to work harder in this setting to show that long limit models are isomorphic. 

\begin{question}
Does Bumby's result hold for injective S-acts? If the answer is negative, characterize the monoids $S$ such that Bumby's result holds for injective S-acts.\footnote{Observe that it holds if $S$ is a group.}
\end{question}

\begin{lemma}\label{lim-sat} Assume $ \Kk_{\aact}$ is stable in $\lambda$ for $\lambda \geq (\operatorname{card}(S) + \aleph_0)^+$ and  $\mu$ is a regular cardinal with $\mu \leq \lambda$. If for every $B \in \Kk_{\aact}$  and $p \in \gS(B)$ there is $Z \subseteq B$ such that $p$ does not $\dnfs$-fork over $SZ$ and $|Z| < \mu$, then the $(\lambda, \mu)$-limit model is saturated.  
\end{lemma}
\begin{proof}Fix  $\{ M_i : i < \mu\} \subseteq (\Kk_{\aact})_\lambda$ a witness to the fact that $M$ is a $(\lambda, \mu)$-limit model. Let $p \in \gS(B)$ with $\operatorname{card}(S) + \aleph_0 \leq \|  B \| < \lambda$ and $B \leq M$. If $p$ is realized in $B$, then clearly $M$ realizes $p$. So we may assume that $p$ is not realized in $B$.

\underline{Claim 1:} There is $N$ with $B \leq N$ and $\{ n_i : i < \| B\|^+\}$ such that $p = \gtp(n_i/B, N)$ for every $i < \| B\|^+$ and $Sn_i \cap Sn_j \subseteq B$ for every $i \neq j <  \| B\|^+$. 

\underline{Proof of Claim 1.}  Let $p = \gtp(a/B, A)$ for $A \in (\Kk_{\aact})_{\| B \|}$ and $a \in A \backslash B$. We build $\{ N_i : i <  \| B\|^+\} \subseteq (\Kk_{\aact})_{\leq \lambda}$ an increasing continuous chain and $\{ n_i : i < \| B\|^+\}$ by recursion such that:

\begin{enumerate}
\item  $N_0= B$,
\item  for every $i <  \| B\|^+$, $n_i \in N_{i +1 }$ and $p = \gtp(n_i/ B, N_{i +1 })$,
\item for every $i <  \| B\|^+$, $Sn_{i} \cap  N_i  \subseteq B$. 
\end{enumerate}

Let $N = \bigcup_{i < \| B\|^+} N_i$. Observe that if $j < i$, then $Sn_i \cap Sn_j \subseteq Sn_i \cap N_{i} \subseteq B$. Therefore, $N$  and $\{ n_i : i < \| B\|^+\}$ are as required. 

We carry out the construction by induction on $i$. The base step is given by Condition (1) and we take unions in limit stages, so we only need to consider the successor step. Assume $j = i +1$.  Then applying the disjoint amalgamation property there are $N_{i+1} \in (\aact)_\lambda$ and $f$ such that the following diagram commutes:

  \[
 \xymatrix{\ar @{} [dr]  A \ar[r]^{ f}  & N_{i+1} \\
 B \ar[u]^{\leq} \ar[r]^{\leq} &  N_i \ar[u]_{\leq} 
 }
\]
  
  and $f[A] \cap N_i = B$. Let $n_{i} = f(a) \in  N_{i+1}$. Observe that $Sn_i \cap N_i  \subseteq f[A] \cap N_i = B$. Hence $N_{i}$ and $n_i \in  N_{i+1}$ are as required.   $\dagger_{\text{Claim 1}}$

Let $N$ and $\{ n_i : i < \| B\|^+\}$ be  as given by Claim 1. By assumption, there is  $Z \subseteq B$ such that $p$ does not $\dnfs$-fork over $SZ$ and $|Z| < \mu$. As $\mu$ is regular, there is an $i < \mu$ such that $SZ \leq M_i$. Then by the amalgamation property there are $f$ and $N^* \in (\aact)_\lambda$ such that the following diagram commutes:

  \[
 \xymatrix{\ar @{} [dr]  SZ \cup S\{n_i :  i < \| B\|^+\}  \ar[r]^{\hspace{1.3cm} f}  & N^* \\
SZ \ar[u]^{\id} \ar[r]^{\id} &  M_i  \ar[u]_{\id} 
 }
\]

As $M_{i+1}$ is universal over $M_i$, there is $g: N^* \xrightarrow[SZ]{} M_{i+1}$. Let $h = g\circ f: SZ \cup S\{n_i :  i < \| B\|^+\} \xrightarrow[SZ]{} M_{i+1}$.

\underline{Claim 2:} There is $i < \| B\|^+$ such that $Sh(n_i) \cap B \subseteq SZ$.

\underline{Proof of Claim 2.}  Assume for the sake of contradiction that this is not the case. Then for every $i <  \| B\|^+$, there is $\ell_i \in Sh(n_i) \cap B$ and $\ell_i \notin  SZ$. Then there are $i\neq j  < \| B\|^+$ such that $\ell_i = \ell_j$. As $\ell_i \in Sh(n_i)$ and $\ell_j \in Sh(n_j)$, there are $s_i, s_j\in S$ such that $h(s_in_i)= \ell_i = \ell_j = h(s_j n_j)$. As $h$ is injective, $s_in_i = s_j n_j$. Then by Claim 1 it follows that  $s_in_i \in B$. Hence $s_in_i \in B \cap Sn_i$. As $ B \cap Sn_i \subseteq SZ$ by Lemma \ref{bar-nf}, because $p$ does not $\dnfs$-fork over $SZ$ and $n_i$ realizes $p$, it follows that  $s_in_i \in SZ$. Hence $\ell_i = s_i h(n_i) = h(s_i n_i) = s_i n_i \in SZ$ where the last equality follows from the fact that $h$ fixes $SZ$. This contradicts the choice of $\ell_i$.  $\dagger_{\text{Claim 2.}}$

Let $i < \| B\|^+$ be as given in Claim 2. Observe that $p, \gtp(h(n_i)/B, M) \in \gS(B)$, $p$  does not $\dnfs$-fork over $SZ$ by assumption, and  $\gtp(h(n_i)/B, M)$ does not $\dnfs$-fork over $SZ$ by Claim 2. Furthermore, $p\upharpoonright {SZ} =  \gtp(h(n_i)/B, M) \upharpoonright {SZ}$ by Claim 1 and as $h\upharpoonright SZ = \id_{SZ}$. Then  $p =  \gtp(h(n_i)/B, M)$ by uniqueness of non-forking extensions (Fact \ref{uniq}). Therefore, $M$ realizes $p$. \end{proof}

\begin{cor}\label{lim-sat2}
Assume $\Kk_{\aact}$ is stable in $\lambda$ for $\lambda \geq (\operatorname{card}(S) + \aleph_0)^+$. If   $\mu$ is a regular cardinal such that $\gamma_r(S) \leq \mu \leq \lambda$, then the $(\lambda, \mu)$-limit model is saturated.  
\end{cor}
\begin{proof}
Follows from Corollary \ref{nf-gtp} and Lemma \ref{lim-sat}.
\end{proof}

\begin{cor}\label{long-lim} Assume $\Kk_{\aact}$ is stable in $\lambda$ and $\lambda \geq (\operatorname{card}(S) + \aleph_0)^+$. 
  If $A$ is a $(\lambda, \alpha)$-limit model, $B$ is a $(\lambda, \beta)$-limit model and $\operatorname{cf}(\alpha),\operatorname{cf}(\beta) \geq \gamma_r(S)$, then $A$ and $B$ are isomorphic. 
\end{cor}
\begin{proof}
Follows from Corollary \ref{lim-sat2} and Fact \ref{sat-iso}.
\end{proof}

We turn our attention to short limit models. We will say that an ideal $I$ is \emph{strictly $\kappa$-generated} if it is generated by $\kappa$ elements but can not be generated by fewer than $\kappa$ elements. 

The following result can be obtained as in \cite[3.4]{mlim}. 

\begin{lemma}\label{non-inj} Assume $\Kk_{\aact}$ is stable in $\lambda$ and $\kappa \leq \lambda$.
If $I$ is a strictly $\kappa$-generated ideal, then the $(\lambda, \kappa)$-limit model is not $\kappa^+$-injective. 
\end{lemma}

The following notion is a parametrized version of weakly noetherian  monoids. Weakly noetherian monoids (sometimes called noetherian monoids) have been thoroughly studied within semigroup theory, see for example \cite{mil}, \cite[\S IV.3.5]{kkm}.

\begin{defin} Assume $\kappa$ is a cardinal. 
A monoid $S$ is \emph{weakly (left) $(<\kappa)$-noetherian} if every (left) ideal is generated by fewer that $\kappa$ elements. If $\kappa = \aleph_0$, we say that $S$ is \emph{weakly  noetherian}. 
\end{defin}

Observe that if $S$ is weakly $(<\kappa)$-noetherian, then $\gamma(S)\leq \kappa$. 

We characterize superstability algebraically. This provides
additional evidence to the already strong case  \cite{maz1}, \cite{m4}, \cite{maz2}, \cite{mj} that superstability is a natural algebraic property.

\begin{theorem}\label{equiv} Assume $S$ is a monoid. 
$S$ is weakly left noetherian if and only if $ (S\text{-acts}, \leq)$ is superstable.
\end{theorem}
\begin{proof}
$\Rightarrow$: Since $\gamma(S) = \gamma_r(S)= \aleph_0$, it follows that $\Kk_{\aact}$ has a unique $\lambda$-limit model for every  $\lambda \geq   2^{\operatorname{card}(S) + \aleph_0}$ by Lemma \ref{sta} and Corollary \ref{long-lim}.

 $\Leftarrow$: Assume for the sake of contradiction that $S$ is not weakly left noetherian. Let $\lambda \geq  2^{\operatorname{card}(S) + \aleph_0}$ such that $\Kk_{\aact}$ has uniqueness of $\lambda$-limit models. In particular the $(\lambda, \omega)$-limit model is isomorphic to the $(\lambda, \omega_1)$-limit model. So the $(\lambda, \omega)$-limit model is $\aleph_1$-injective by Lemma \ref{l-inj}. On other hand, since $S$ is not weakly left noetherian there is $I$ a strictly $\aleph_0$-generated ideal. Hence the $(\lambda, \omega)$-limit model is not $\aleph_1$-injective by Lemma \ref{non-inj}. This is clearly a contradiction. 
 \end{proof}
 
 \begin{remark}
 
 The result above is an extension of the following result for modules \cite[3.12]{maz1}:   Assume $R$ is a ring. 
$R$ is left noetherian if and only if the class of $R$-modules with embeddings  is superstable. The  argument in  \cite{maz1} is very different to the one given here, but the result can be obtained by a similar argument (see  \cite[3.17]{mlim}).
 \end{remark}

 \begin{remark}\label{first-ext} A classical result of the first-order model theory of acts \cite{iva}, \cite[3.5]{fogo} is that if $S$ is left coherent then:
 \begin{itemize}
 \item the model companion of the theory of $S$-acts is stable;
 \item $S$ is weakly left noetherian if and only if the model companion of the theory of $S$-acts is superstable. 
\end{itemize}  

These results follow directly from Theorem \ref{sta-2} and Theorem \ref{equiv}. Since if $S$ is left coherent, then  there is a bijection between the first-order complete types in the model companion of the theory of $S$-acts and the Galois types of $\Kk_{\aact}$. This follows from Fact \ref{b-types}, quantifier-elimination of the model companion and the fact that every act can be embedded into a model of the model companion. 

Therefore,  Theorem \ref{sta-2} and Theorem \ref{equiv} can be seen as an extension of  \cite{iva}, \cite[3.5]{fogo}. It is worth mentioning that the methods used in \cite{iva}, \cite[3.5]{fogo} are very different to the methods used in this paper.
 \end{remark}
 
  \begin{remark}\label{lim=ss}
 Using Theorem \ref{sta-2} one can show that $S$ is weakly left noetherian if and only if $ (S\text{-acts}, \leq)$ is stable in a tail of cardinals. The forward direction is clear. As for the backward direction, if $S$ is not weakly noetherian then $\gamma(S) > \aleph_0$ and since there are arbitrarily large cardinals of cofinality $\aleph_0$, it follows from K\"{o}nig's Lemma that there are arbitrarily large cardinals where  $ (S\text{-acts}, \leq)$ is unstable.
 \end{remark}

We provide a few examples of when  $ (S\text{-acts}, \leq)$ is superstable.

\begin{cor} If $S$ is a finite monoid, $S$ is a group or $S$ is a finitely generated commutative monoid, then $ (S\text{-acts}, \leq)$ is superstable.
\end{cor}
\begin{proof}
It is clear that finite monoids and groups are weakly noetherian and it follows from  \cite{re} that finitely generated commutative monoids are weakly noetherian. Hence in each case we have that $ (S\text{-acts}, \leq)$ is superstable by Theorem \ref{equiv}. 
\end{proof}

\begin{remark}\label{compare}
Recall that a monoid $S$ is right reversible if $St \cap Sr \neq \emptyset$ for every $t, r \in S$. Moreover, $S$ is right reversible if and only if all coproducts of injective left acts are injective (see for example \cite[\S III.1.13]{kkm}). Therefore any weakly left noetherian monoid which is not a right reversible monoid satisfies that  $ (S\text{-acts}, \leq)$ is superstable but there are coproducts of injective left acts which are not injective.  An explicit example for $S$ is  $\{ \begin{pmatrix} 1 & x \\ 0 & 0  \end{pmatrix} : x \in \mathbb{Z}_{2} \}^1$ with matrix products where the $1$ stands for adding an identity.

The previous  example shows that the work done in \cite[\S 5]{mj} for classes of modules cannot be extended to this setting putting a bound on how far the arguments given in \cite[\S 5]{mj} can be extended. This partially solves the problem stated in the last line of \cite{mj} which asks to determine the relationship between superstability and closure under coproducts of injective objects in arbitrary categories.
\end{remark}

As in \cite[\S 3]{mlim}, it is possible to extend Theorem \ref{equiv} to the strictly stable setting. Since the arguments of \cite[\S 3]{mlim} work for acts, we only record the main result which is an extension of  \cite[3.17]{mlim} from modules to acts. 

\begin{theorem}\label{main-f} Let $n \geq 0$. The following are equivalent.

\begin{enumerate}
\item $\gamma(S)= \aleph_n$.
\item $ (S\text{-acts}, \leq)$ has exactly $n +1 $ non-isomorphic $\lambda$-limit models for every $\lambda \geq 2^{\operatorname{card}(S) + \aleph_0}$ such that $\Kk_{\aact}$ is stable in $\lambda$. 
\end{enumerate}
\end{theorem}

 \begin{remark} It is possible to provide a proof of Theorem \ref{equiv} and Theorem \ref{main-f} using heavy model theoretic machinery. More precisely, one can obtain the results using Lemma \ref{g=k} and  \cite[4.8, 6.1, 5.1]{limit25}. As the arguments of this paper are significantly shorter than the general arguments we decided to avoid the abstract machinery. Furthermore, the proofs used in this paper were obtained prior to \cite{limit25}.
 
 \end{remark}

\subsection{Algebraic results} We present a couple of algebraic results that follow from the study of limit models. 

The first result extends \cite[3.5, 3.6]{mlim} from modules to acts. As the proof is analogous to that for modules, we only sketch the proof.

\begin{theorem}\label{inj-equiv}Assume $\kappa$ is a regular cardinal. The following are equivalent.
\begin{enumerate}
\item $S$ is weakly left  $(<\kappa)$-noetherian.
\item  Every $\kappa$-injective act is weakly injective. 
\item Every $\kappa$-injective act is $\kappa^+$-injective. 
\item The $(\lambda, \kappa)$-limit model is $\kappa^+$-injective for every $\lambda$ such that $\Kk_{\aact}$ is stable in $\lambda$ and $ \kappa \leq \lambda$. 
\end{enumerate}
\end{theorem}
\begin{proof}
$(1) \Rightarrow (2) \Rightarrow (3)$ are clear and $(3) \Rightarrow (4)$ follows from Lemma \ref{l-inj}. Finally, $(4) \Rightarrow (1)$ follows from Lemma \ref{non-inj}. \end{proof}

\begin{remark}
The  result above also generalizes  \cite[3.7]{g-co} which obtains  the equivalence of (1), (2) and (3) for  $\kappa = \aleph_0$, i.e., for weakly noetherian monoids. The methods used in this paper are completely different to those of \cite{g-co}.
\end{remark}

Recall that an act $A$ is \emph{absolutely pure} if for every $B \leq C$ such that $B$ is finitely generated and $C$ finitely presented  and $f: B \to A$ an $S$-homomorphism there is an $S$-homomorphism $g : C \to A$ such that $g(b) = f(b)$ for every $b \in B$ \cite[3.8]{g-c}. 

\begin{lemma}\label{o-abs}
Assume $\lambda \geq \operatorname{card} (S)+ \aleph_0$ and $\delta < \lambda^+$ is a limit ordinal.
If $A$ is a $(\lambda, \delta)$-limit model, then $A$ is absolutely pure.
\end{lemma}
\begin{proof} Let $B \leq C$ such that $B$ is finitely generated, $C$ is finitely presented  and $f: B \to A$ is an S-homomorphism. Fix $\{ A_i : i < \delta \}$ a witness to the fact that $A$ is a $(\lambda, \delta)$-limit model. As $B$ is finitely generated there is an $i < \delta$ such that $f[B] \subseteq A_i$. As $C \in (\Kk_{\aact})_{\operatorname{card}(S) + \aleph_0}$ and $A_i \in (\Kk_{\aact})_\lambda$, then taking the pushout of $B \leq C$ and $f: B \to A_i$ there are $f_1: C \to P$ an S-homomorphism and $f_i: A_i \to P$ an $S$-monomorphism such that $P \in  (\Kk_{\aact})_\lambda$ and $f_i \circ f = f_1 \rest B$. Then there is $Q \in  (\Kk_{\aact})_\lambda$ and $h: Q \to P$ an isomorphism  with $A_i \leq Q$ and $h \rest A_i = f_i$. So we have that the following diagram commutes:

  \[
 \xymatrix{\ar @{} [dr]  C \ar[r]^{h^{-1}\circ f_1}  & Q \\
 B \ar[u]^{\leq} \ar[r]^{f} &  A_{i} \ar[u]_{\leq } 
 }
\]
Since $A_{i+1}$ is universal over $A_i$ there is $f_2 : Q \xrightarrow[A_i]{} A_{i+1}$. It is easy to show that $g = f_2 \circ h^{-1} \circ f_1 : C \to A $ is as required.
\end{proof}

The next result provides a characterization of weakly noetherian monoids using absolutely pure acts extending a classical result of ring theory \cite[Theorem 3]{meg}.

\begin{cor}\label{abs} Assume $S$ is a monoid.
Every absolutely pure $S$-act is weakly injective if and only if $S$ is weakly left noetherian.
\end{cor}
\begin{proof}

$\Rightarrow$:  Let $A$ be the  $(\lambda, \omega)$-limit model for $\lambda$ such that $\Kk_{\aact}$ is stable in $\lambda$. Then $A$ is absolutely pure by Lemma \ref{o-abs}. Hence $A$ is weakly injective by assumption. Therefore, $S$ is weakly left noetherian by Theorem \ref{inj-equiv}.(4). 

$\Leftarrow$: Let $A$ be an absolutely pure act. Then $A$ is $\aleph_0$-injective as $S$ is a finitely presented left $S$-act. Hence $A$ is weakly injective by Theorem \ref{inj-equiv}.(2).
\end{proof}

\begin{remark}\label{dif-mod}
The result above is an extension of the following classical result for rings \cite[Theorem 3]{meg}: every absolutely pure module is injective if and only if $R$ is left noetherian. There are two key differences: the equivalence for acts substituting \emph{weakly injective} for \emph{injective} fails as noticed in \cite[p. 168]{normak} and being weakly injective  does not imply absolute purity even if $S$ is weakly noetherian \cite[p. 213]{kkm}.

Moreover, the argument for modules is completely different as it relies on the fact that a ring is noetherian if and only if injective modules are closed under direct sums. As mentioned earlier this result does not hold for acts (see Remark \ref{compare}).  

Furthermore, the previous result  also extends the following result for acts \cite[Proposition 8]{normak} by providing an equivalence by replacing \emph{injective} by \emph{weakly injective} in the statement of \cite{normak}. 
\end{remark}
\appendix
\section{}


In this section, we will show that some of our main results can be obtained using methods from category theory. 

Monomorphisms of $S$-acts are closed under \textit{transfinite compositions}. This means that given a continuous chain 
$(f_{ij}\colon A_i \to A_j)_{i\leq j< \lambda}$ of monomorphisms where $\lambda$ is a limit ordinal then a colimit of $(f_{0i}\colon A_0 \to A_i)_{i<\lambda}$ is a monomorphism. Recall that a chain
$(f_{ij}\colon A_i \to A_j)_{i\leq j< \lambda}$ is continuous
if $f_{i\kappa}$ is a colimit of $(f_{ij}\colon A_i \to A_j)_{i\leq j< \kappa}$ 
for every limit ordinal $\kappa <\lambda$.

It follows from the previous paragraph and Section 2.1 that monomorphisms form a \emph{cellular class} in $S$-$\Act$, i.e.,
a class closed under pushouts along homomorphisms and under transfinite compositions.

 A class $\cx$ of monomorphisms \textit{cellularly generates}
if its closure under pushouts and transfinite compositions is the class of all monomorphisms.

\begin{lemma}\label{cofgen2} 
Monomorphisms of $S$-$\Acts$ are cellularly generated by
$A\to B$ with $B$ cyclic.
\end{lemma}
\begin{proof}
Let $f:K\to L$ be a monomorphism of $S$-acts. We can assume $f$ to be an inclusion. Consider $a\in L\setminus K$. It corresponds to $h:S\to L$ where $h(s)=sa$ for every $s \in S$. Let
$$
S \xrightarrow{\ h_1 \ } B \xrightarrow{\ h_2 \ } L 
$$
be the (epimorphism, monomorphism) factorization of $h$. Then $B$ is cyclic. Consider the pullback
$$
		\xymatrix@=2pc{
				A \ar [r]^{g}\ar[d]_{u} & B \ar[d]^{h_2}\\
			K  \ar [r]_{f}& L
		}
		$$ 
and the pushout
$$
		\xymatrix@=2pc{
				A \ar [r]^{g}\ar[d]_{u} & B \ar[d]^{v}\\
			K  \ar [r]_{f_{01}}& K_1
		}
$$ 
Since $S$-$\Acts$ has effective unions by Remark \ref{e-squares}, the induced morphism $t:K_1\to L$ is a monomorphism. Hence $h=tvh_1$. Continuing this procedure (in limit steps we take colimits), we get $f$ as a transfinite composition of pushouts of monomorphisms $A\to B$
with $B$ cyclic.
\end{proof}

\begin{remark}
Since the class of monomorphisms to which an object is injective is closed under pushouts along homomorphisms and under transfinite compositions, Lemma \ref{cofgen2} implies the Skornjakov--Baer criterion for injectivity (see for example \cite[\S III.1.8]{kkm}).
\end{remark}

We obtain an slightly stronger result than Corollary \ref{sub-cyc}. 

\begin{lemma}\label{gen}
If $A$ is a subact of $B$ and $B$ is cyclic, then $A$ is generated by fewer than $\gamma(S)$ elements.

\end{lemma}
\begin{proof}
Let $B$ be a cyclic act. Hence it admits a surjective homomorphism $f:S\to B$.
Consider a subact $u:A\to B$ and take the pullback
$$
		\xymatrix@=2pc{
				P \ar [r]^{\bar{f}}\ar[d]_{\bar{u}} & A \ar[d]^{u}\\
			S  \ar [r]_{f}& B
		}
		$$ 
Then $\bar{u}$ is a monomorphism and, since $S$-$\Act$ is regular, $\bar{f}$ is surjective. Hence $A$ is generated by fewer than $\gamma(S)$ elements as a quotient of  $P$.
\end{proof}

\begin{remark}\label{soa}
The embedding of an object $K$ to an injective object can be obtained by the small object argument. It is given by a transfinite composition of a continuous chain $(h_{ij})$ of monomorphisms where
$K_0=K$ and $h_{ii+1}$ is given by the pushout 
$$
		\xymatrix@=2pc{
				\coprod_{(u,f)} A_{(u,f)} \ar [r]^{}\ar[d]_{} & 
\coprod_{(u,f)} B_{(u,f)} \ar[d]^{}\\
			K_i  \ar [r]_{h_{ii+1}}& K_{i+1}
		}
$$ 
given by spans
$$
		\xymatrix@=2pc{
				A_{(u,f)} \ar [r]^{f}\ar[d]_{u} & B_{(u,f)} \\
			K_i  
		}
$$ 
where $f:A_{(u,f)}\to B_{(u,f)}$ are monomorphisms between cyclic acts and $u:A_{(u,f)}\to K_i$ are homomorphisms (see \cite[12.2.2]{r}). This pushout can be replaced by a transfinite composition of pushouts of corresponding spans. Or,
by a multipushout of all the spans above (see 
\cite[II.4 and II.7]{ahrt}). 
Pushouts of corresponding spans will be called parts of the multipushout.
\end{remark}

Lemma \ref{cofgen2} yields another proof of Lemma \ref{sta}.

\begin{proof}[Proof of Lemma \ref{sta}]
Following Remark \ref{limit1}, we have to prove that $\Kk_{\aact}$ has a $\lambda$-limit model. For this, it suffices to show that
$\Kk_{\aact}$ has a universal object $K^\ast$ of size $\lambda$ over every object $K$ of size $\lambda$. We will prove that $h:K\to K^\ast$ is given by a small object argument of length $\lambda$.
We will use multipushouts like in \cite[II.7]{ahrt} (see Remark 
\ref{soa}). Hence $h$ is a transfinite composition of a chain of monomorphisms $(h_{ij}:K_i\to K_j)_{i<j<\lambda}$ where $K_0=K$.

Since cyclic acts are given by congruences on $S$, the number of cyclic acts is $\leq 2^{\operatorname{card} (S)+\aleph_0}$. A cyclic act
has $\leq 2^{\operatorname{card}(S)}$ subacts. Hence the number of subacts $A$ of cyclic acts is 
$$
\leq 2^{\operatorname{card}(S)+\aleph_0}\cdot2^{\operatorname{card}(S)}
\leq 2^{\operatorname{card}(S)+\aleph_0}\leq\lambda.
 $$
Since every subact $A$ of a cyclic act $B$ is generated by $<\gamma(S)$ elements by Corollary \ref{gen}, the number of morphisms $A\to K$ is $\leq\lambda$. Since multipushouts in 
the construction do not increase sizes, the size of $K^\ast$ is $\lambda$. Moreover, $h$ is a monomorphism. We have to prove that for every monomorphism $f:K\to L$ where $L$ is of size $\lambda$ there is a monomorphism $g:L\to K^\ast$ such that $gf=h$.

Following Lemma \ref{cofgen2}, $f$ is a transfinite composition of a chain of monomorphisms $(f_{ij}:L_i\to L_j)_{i<j<\lambda}$ where $L_0=K$ and $f_{ii+1}$ is a pushout
$$
		\xymatrix@=2pc{
				A \ar [r]^{u}\ar[d]_{v} & B \ar[d]^{\bar{v}}\\
			L_i  \ar [r]_{f_{ii+1}}& L_{i+1}
		}
		$$ 
with $B$ cyclic. For $i=0$, this pushout is a part of a multiple pushout giving the first step $h_{01}:K\to K_1$ of the small object argument. Thus the induced morphism $g_1:L_1\to K_1$ is a monomorphism (see
\cite[Remark II.2 and the proof of II.7]{ahrt}). Next, for $i=1$, we have
pushouts
$$
		\xymatrix@=2pc{
		A \ar [r]^{u}\ar[d]_{v} & B \ar[d]^{\bar{v}}\\
				L_1 \ar [r]^{f_{12}}\ar[d]_{g_1} & L_2 \ar[d]^{\bar{g}_1}\\
			K_1  \ar [r]_{\bar{f}_{12}}& P
		}
		$$  
where the outer rectangle is a part of a multiple pushout giving the second step $h_{12}:K_1\to K_2$ of the small object argument.
Since $\bar{g}_1$ is a monomorphism because $g_1$ is a monomorphism
and the induced morphism $g'_2:P\to K_2$ is a monomorphism by \cite[Remark II.2 and the proof of II.7]{ahrt}, 
the composition $g_2=g'_2\bar{g}_1:L_2\to K_2$ is a monomorphism. By a transfinite iteration we get a monomorphism $g:L\to K^\ast$ such that $gf=h$.
\end{proof}


\begin{remark}\label{sat}
The object $K^\ast$ from the proof above is $\lambda$-saturated. 
For this, it suffices to show that, for every monomorphisms $f:A\to K^\ast$ and $g:A\to B$ where $A$ and $B$ have size $<\lambda$, there
is a monomorphism $f':B\to K^\ast$ such that $f'g=f$ (see \cite{lr}).
However, $f$ factorizes through some $K_i\to K$ by $f_1:A\to K_i$, which yields $f'':B\to K_{i+1}$ such that $h_{ii+1}f_1=f''g$. Then $f'$ is the composition of $K_{i+1}\to K^\ast$ with $f''$.
\end{remark}

We will give another proof of Theorem \ref{equiv}.
\begin{theorem-3}
Assume $S$ is a monoid. 
$S$ is weakly left noetherian if and only if $ (S\text{-acts}, \leq)$ is superstable.
\end{theorem-3}

\begin{proof}
$\Rightarrow$: If $S$ is weakly noetherian then, $\Kk_{\aact}$ is stable in $\lambda$ for every $\lambda \geq 2^{2^{\operatorname{card} (S)+ \aleph_0}}$ by Lemma  \ref{sta}. 
Following Remark \ref{tail}, $\Kk_{\aact}$ is superstable.

$\Leftarrow$:  Assume $\Kk_{\aact}$ is superstable and assume for the sake of contradiction that $S$ is not weakly left noetherian. Then there is a left ideal $A$ of $S$ which is not finitely generated. Thus there are $\{a_i:i<\omega\}$ in $A$ such that for every $i < \omega$,  $a_{i+1} \not\in S\{a_0,\dots,a_i\}$. 

Let $\lambda > \LS(\Kk_{\aact})$ be sufficiently large such that $\lambda$-saturated objects are closed under directed colimits in  $\Kk_{\aact}$. Such a $\lambda$ exists by \cite[1.3]{grva} since $\Kk_{\aact}$ is superstable.

We can construct $\{ B_i : i<\omega \}$ as in Lemma \ref{g=k} but replacing Condition (1) of that construction by $B_{i+1}=(B_i)^\ast$.
Following Remark \ref{sat}, for every $i >0$, $B_i$ is $\lambda$-saturated. Since $\lambda$-saturated objects are closed under directed colimits, $B_\omega = \bigcup_{i < \omega} B_i$ is $\lambda$-saturated. Thus there exists 
$f:S\to B_\omega$ whose restriction to 
$S\{a_i : i < \omega\}$ is the inclusion of $S\{a_i : i < \omega\}$ 
to $B_\omega$. Then there is $i<\omega$ such that $f(1)\in B_i$. Hence 
$$
a_{i+1}=f(a_{i+1})=f(a_{i+1}1)=a_{i+1}f(1)\in S\{a_0,\dots,a_i\},
$$ 
which is a contradiction to the choice of $a_{i+1}$. \end{proof}

\end{document}